\documentclass[a4paper,10pt,thmsa, reqno]{amsart}
\usepackage{amsmath, latexsym, amsfonts, amssymb, amsthm, amscd}
\usepackage{color}
\usepackage{color}

\usepackage{ulem}

\newcommand{\bR}{\mathbf{R}}
\newcommand{\dd}{\mathrm{d}}

\newcommand{\sF}{\mathcal{F}}

\newcommand{\wt}{\widetilde}

\newcommand{\eil}{\overset{\text\tiny\rm{(law)}}{=}}

\theoremstyle{thm}
  \newtheorem{thm}{Theorem}
  \newtheorem{cor}[thm]{Corollary}

\theoremstyle{def}

\theoremstyle{remark}
  \newtheorem{rem}{Remark}

\title{Further studies on square-root boundaries for Bessel processes}
\author{Larbi Alili and Hiroyuki Matsumoto}
\address{L.~Alili -- Department of
Statistics, The University of Warwick, CV4 7AL, Coventry, UK.}
\email{L.Alili@warwick.ac.uk}
\address{H. ~Matsumoto -- Department of Physics and Mathematics,
Aoyama Gakuin University,
Fuchinobe 5-10-1, Chuu-ouku, Sagamihara 252-5258, Japan.}
\email{matsu@gem.aoyama.ac.jp}


\date{\today}

\begin{document}
\begin{abstract}
We  look at decompositions of perpetuities and apply them to the study of the distributions of hitting times of Bessel processes of two types of square root boundaries. These distributions are linked giving a new proof of some Mellin transforms results obtained by DeLong \cite{delong-83} and Yor \cite{Yor-sqrt-bdry}. Several random factorizations and characterizations of the studied distributions are established.
\end{abstract}

\maketitle
\section{Introduction}
Let  $R:=((R_t)_{t\geqq 0}, P_x^{(\mu)})$  be a Bessel process of
index $\mu \in \mathbf{R}$, or dimension $\delta=2(1+\mu)$, started at $x$.  Assume that $R$ is killed when it hits $0$ so that the life time $\zeta$ of $R$ equals the first hitting time $T_0^R$ of 0 by $R$, i.e. $T_0^R=\inf\{ s>0; R_s=0\}$. Here and below, unless otherwise specified, we assume that $\inf\emptyset =+\infty$. Recall that if $\mu\geq 0$ then 0 is polar (entrance-not-exit) for $R$ and hence $\zeta=+\infty$ a.s. If either $-1<\mu<0$ ($0$ is non-singular) or $\mu\leq -1$ ($0$ is exit-not-entrance), see (\cite{Borodin-Salminen}, P.133),  then $\zeta<\infty$ a.s. Thus, we have
\begin{align*}
 \zeta = \left\{
  \begin{array}{l l}
    T_0^R & \quad \text{if $\mu < 0$;}\\[4pt]
    +\infty & \quad \text{if $\mu\geq 0$.}
  \end{array} \right.
\end{align*}
 Next, for $b$ and $c>0$, we consider the first hitting time{\bf s} $\sigma_{\pm}$
of the square-root boundaries $s\mapsto \sqrt{(b\pm s)/c}$,
defined by

\begin{equation}\label{def-sigma-plus}
 \sigma_{+}=\inf\{s>0\; ;\; R_s=\sqrt{(b+ s)/c}\}
\end{equation}
and
\begin{equation}\label{def-sigma-minus}
 \sigma_{-}=\inf\{s<b\; ;\; R_s=\sqrt{(b- s)/c}\}, \quad \inf{\emptyset}=b.
\end{equation}
The Mellin transform of the distribution of $\sigma_{+}$ has been computed
in \cite{Yor-sqrt-bdry} and reads, with $\mu=\pm \nu$  for $\nu>0$, as
\begin{align}
 & E^{(-\nu)} \Big[ \Big(1+\frac{\sigma_+}{b} \Big)^{-a} \Big] =
\frac{b^{\nu}\Lambda(\nu+a, \nu+1, \frac{1}{2b})}
{c^{-\nu}\Lambda(\nu+a, \nu+1, \frac{1}{2c})}  \label{1e:by-yor-minus} \\
\intertext{and}
 & E^{(\nu)} \Big[  \Big(1+\frac{\sigma_+}{b}\Big)^{-a} \Big] =
\frac{\Lambda(a, \nu+1, \frac{1}{2b})}{\Lambda(a, \nu+1, \frac{1}{2c})}
\label{1e:by-yor-plus}
\end{align}
where $E^{(\pm\nu)}:=E^{\pm\nu}_1$ is the expectation with respect to $P_1^{(\pm\nu)}$ and
$\Lambda$ is the confluent hypergeometric function $\Psi$ when $b<c$
and is $\Phi$ when $b>c$; $\Psi$ and $\Phi$ have the integral representations
\begin{equation}\label{Int-Psi}
 \Psi(\alpha, \beta, z)=\frac{1}{\Gamma(\alpha)}
\int_0^{\infty} e^{-zt} t^{\alpha-1} (1+t)^{\beta-\alpha-1} \dd t,
\qquad \alpha>0,
\end{equation}
and
\begin{equation}\label{Int-Phi}
  \Phi(\alpha,\beta;z) =
\frac{\Gamma(\beta)}{\Gamma(\alpha)\Gamma(\beta-\alpha)}
\int_0^1 t^{\alpha-1}(1-t)^{\beta-\alpha-1} e^{zt} \dd t,
\quad 0<\alpha<\beta,
\end{equation}
found in (\cite{Lebedev}, p.266, (9.11.1) and p.268, (9.11.6)).

Formulae \eqref{1e:by-yor-minus} and \eqref{1e:by-yor-plus} are probabilistically proved in \cite{ESY}, in case $b<c$, by using Lamperti's relation to relate $\sigma_{+}$ to the first passage times at constant levels of the diffusion
\begin{equation*} \big( e^{-2B_t^{(-\nu)}} (b+A_t^{(-\nu)}), t\geqq 0\big)
\end{equation*}
where $(B_t^{(-\nu)}, t\geqq 0)$ is a Brownian motion with constant drift $-\nu$ and
\begin{equation}\label{exponential}
A_t^{(-\nu)}=\int_0^{t}e^{B_s^{(-\nu)}}\;ds,\; t>0.
\end{equation}


The aim of this article is to provide characterisations of  the distributions of $\sigma_{-}$ and $\sigma_+$ through Mellin transforms and some random factorizations. We establish the analogues of the results in \cite{ESY} for the distribution of $\sigma_{-}$. Then, exploiting ideas of \cite{alili-patie}, we give a relationship that relates the distributions of $\sigma_{-}$ and $\sigma_{+}$. This is applied to get a new proof of a result which appeared in (\cite{delong-81}, \cite{delong-83}, \cite{Yor-sqrt-bdry}) that gives the Mellin transform of the distribution of  $\sigma_{+}$  in case $b>c$ and $\delta\geqq 2$. Various random equations  and characterizations are given for $\sigma_{\pm}$ by considering separately four cases corresponding to the different possible signs of $\nu$ and $b-c$; see  (\cite{Goldie}, \cite{Kesten}) for random affine equations.  For simplicity, we worked with Bessel processes but some of our results extend readily to positive spectrally one sided self-similar Markov processes with index 2, see  Remark \ref{Remark-3}.

\section{Characterization of the distribution of $\sigma_{-}$}

Let $B^{(-\nu)}=((B^{(-\nu)}_t)_{t\geqq0}, P^{(-\nu)})$ be a Brownian motion
with constant drift $-\nu$ starting at $0$. Recall that
$A_t^{(-\nu)}$, $t\geqq 0$,  is the exponential functional defined by \eqref{exponential}.
Then, by Lamperti's relation, there exists a Bessel process $R$
starting at $1$ such that
\begin{equation}\label{Lamperti}
e^{B_t^{(-\nu)}}=R_{A_t^{(-\nu)}},\; t\geqq 0.
\end{equation}
In particular, the life time of $R$ is given by $\zeta=T_0^R=A_{\infty}^{(-\nu)}$.
In the sequel, we suppose that $\nu\geqq 0$ and $-\nu$ will always refer to a non-positive  constant drift for Brownian motions and non-positive  index for Bessel processes.
A variant of Dufresne's identity states that
\begin{equation*}
A_\infty^{(-\nu)} := \lim_{t\to\infty}A_t^{(-\nu)} \eil
\frac{1}{2\Gamma_\nu}
\end{equation*}
for a gamma random variable $\Gamma_\nu$ with parameter $\nu$. A variable $Z$ such that $Z\stackrel{(law)}{=}A_\infty^{(-\nu)}$ having the distribution
\begin{equation*}P^{(- \nu)}(Z\in dx)=\frac12\frac{1}{\Gamma(\nu)}\Big(\frac{1}{2x}\Big)^{\nu-1}e^{-\frac{1}{2x}}\; dx, x>0,
\end{equation*}
represents the present value of a perpetuity, see (\cite{dufresne}); $Z$ will always be distributed as above even when defined under $P^{(\nu)}$.
To start with,  we look at the Mellin transform of the distribution of $\sigma_{-}$ and related random factorizations.

\begin{thm} \label{proposition-b-larger-than-c}
The distribution of $\sigma_{-}$ under $P^{(-\nu)}_1$ is
characterized by the following random equations. We have
\begin{equation}\label{rae-sigmaminus}
\Big(\frac{Z}{b}-1\Big)_+\stackrel{(law)}{=}\Big(1-\frac{\sigma_{-}}{b}\Big)(\frac{Z}{c}-1)_+
\end{equation}
if $b>c$ and
\begin{equation}\label{rae-sigmaminus-b-less-c}
\Big(1-\frac{Z}{b}\Big)_{+}\stackrel{(law)}{=}\Big(1-\frac{\sigma_{-}}{b}\Big)(1-\frac{Z}{c})_+
\end{equation}
if $b<c$,
where  $(x)_{+}$ stands for the positive part of $x\in\bR$ and
the variables $\sigma_-$ and $Z$ on the right hand sides are independent.
Furthermore,  for $a<\nu$, we have
\begin{equation} \label{2e:sigma-minus-mellin}
E^{(-\nu)}\Big[ \Big( 1-\frac{\sigma_{-}}{b} \Big)^a \Big]  =
\frac{b^{-\nu}e^{-\frac{1}{2b}}\Lambda(a+1,\nu+1;\frac{1}{2b})}
{c^{-\nu}e^{-\frac{1}{2c}}\Lambda(a+1,\nu+1;\frac{1}{2c})},
\end{equation}
where, as in the previous section, $\Lambda=\Psi$ when $b<c$ and
$\Lambda=\Phi$ when $b>c$.
\end{thm}

\begin{proof} Let us prove (\ref{rae-sigmaminus}) and (\ref{rae-sigmaminus-b-less-c}).
Set

\begin{eqnarray*}
\tau_{b, c} := \inf\{s>0\; ;\; e^{-2B_s^{(-\nu)}}(b-A_s^{(-\nu)})=c\}, \quad \inf \emptyset=C_b^{(-\nu)},
\end{eqnarray*}
 where $C^{(-\nu)}_t:=\int_0^tR_s^{-2}\, ds$ stands for the inverse of $A_t^{(-\nu)}$, $t>0$. Write simply $\tau$ for $\tau_{b,c}$ where there is no risk of confusion. Because
\begin{eqnarray*}
\tau=\inf\Big\{s>0\; ;\; (R_{A_s^{(-\nu)}})^2=
\frac{b-A_s^{(-\nu)}}{c} \Big\},
\end{eqnarray*}
we see that $A_{\tau}^{(-\nu)}=\sigma_{-}$. Next, on the event $\tau<C_{b}^{(-\nu)}$,  we have
\begin{align}\label{intermediate}
A_\infty^{(-\nu)} - b  = A_\tau^{(-\nu)} + e^{2B_\tau^{(-\nu)}}(\wt{Z}-c)
+ c e^{2B_\tau^{(-\nu)}} - b
  = e^{2B_\tau^{(-\nu)}}(\wt{Z}-c)
\end{align}
where
$
\wt{Z} := \int_0^\infty e^{2(B_{\tau+s}^{(-\nu)}-B_\tau^{(-\nu)})} \dd s$. Observe that, by the strong Markov property, $\wt{Z}$ is distributed as $(2\Gamma_\nu)^{-1}$ and
is independent of $(B_s^{(-\nu)}, s\leq \tau)$.\\

Suppose that $b>c$. On the event $\{\tau=C_b^{(-\nu)}\}$,  we have $
b -A_s^{(-\nu)} > c e^{2B_s^{(-\nu)}}$, $s>0$, and hence, since $B^{(-\nu)}$ drifts to $-\infty$, by letting $s$ tend to $+\infty$, we get
 that $A_\infty^{(-\nu)}\leqq b$. Thus,  by using (\ref{intermediate}) and the fact that $\{\tau<C_b^{(-\nu)}\}=\{\sigma_{-}<b \}$, we obtain
 \begin{eqnarray*}
(A_\infty^{(-\nu)}-b)_{+} &=& (A_\infty^{(-\nu)}-b)_{+}{\bf 1}_{\{\tau<C_b^{(-\nu)}\}}+ (A_\infty^{(-\nu)}-b)_{+}{\bf 1}_{\{\tau=C_b^{(-\nu)}\}} \\
&=& e^{2B_\tau^{(-\nu)}}(\wt{Z}-c)_+{\bf 1}_{\{\tau<C_b^{(-\nu)}\}}\\
&=& (\frac{b-A_{\tau}^{(-\nu)}}{c})(\wt{Z}-c)_+{\bf 1}_{\{\tau<C_b^{(-\nu)}\}}\\
&=&(\frac{b-\sigma_{-}}{c})(\wt{Z}-c)_+ {\bf 1}_{\{\sigma_{-}<b\}}.
\end{eqnarray*}

The proof of (\ref{rae-sigmaminus}) is now complete because  $\sigma_{-}\leq b$ holds $P^{(-\nu)}$ a.s.\\

Suppose that $b<c$. On the event $\{\tau=C_b^{(-\nu)}\}$,  we have $
b -A_s^{(-\nu)} < c e^{2B_s^{(-\nu)}}$, $s>0$, and hence,  by letting $s$ tend to $+\infty$, we obtain $A_\infty^{(-\nu)}\geqq b$. Hence,  by using (\ref{intermediate}) and the fact that $\{\tau<C_b^{(-\nu)}\}=\{\sigma_{-}<b  \}$, we obtain
 \begin{eqnarray*}
(b-A_\infty^{(-\nu)})_{+} &=& (b-A_\infty^{(-\nu)})_{+}{\bf 1}_{\{\tau<+C_b^{(-\nu)}\}}+ (b-A_\infty^{(-\nu)})_{+}{\bf 1}_{\{\tau=C_b^{(-\nu)}\}} \\
&=& e^{2B_\tau^{(-\nu)}}(c-\wt{Z})_+{\bf 1}_{\{\tau<C_b^{(-\nu)}\}}\\
&=& (\frac{b-A_{\tau}^{(-\nu)}}{c})(c-\wt{Z})_+{\bf 1}_{\{\tau<C_b^{(-\nu)}\}}\\
&=&(\frac{b-\sigma_{-}}{c})(c-\wt{Z})_+ {\bf 1}_{\{\sigma_{-}<b\}}.
\end{eqnarray*}
The proof of (\ref{rae-sigmaminus-b-less-c}) is complete because  $\sigma_{-}\leq b$ holds $P^{(-\nu)}$ a.s.

By taking the Mellin transform of both sides of (\ref{rae-sigmaminus}) and using
\begin{equation} \label{2e:z-b} \begin{split}
E^{(-\nu)}[(Z-b)_{+}^a] & = \int_0^{\frac{1}{2b}}
\frac{1}{\Gamma(\nu)} \Big(\frac{1}{2x}-b\Big)^a x^{\nu-1} e^{-x} \dd x \\
 & = \frac{b^{a-\nu}}{2^\nu \Gamma(\nu)} e^{-\frac{1}{2b}}
\frac{\Gamma(a+1)\Gamma(\nu-a)}{\Gamma(\nu+1)}
\Phi(a+1,\nu+1;\frac{1}{2b}),
\end{split} \end{equation}
we obtain
\begin{align*}
E^{(-\nu)}\Big[ \Big( \frac{b-\sigma_{-}}{c} \Big)^a \Big]
&=\frac{E^{(-\nu)}\left[(A_\infty^{(-\nu)}-b)_{+}^a\right]}{E^{(-\nu)}\left[(A_\infty^{(-\nu)}-c)_{+}^a\right]}\\
 & = \frac{b^{a-\nu} e^{-\frac{1}{2b}} \Phi(a+1,\nu+1;\frac{1}{2b})}
{c^{a-\nu} e^{-\frac{1}{2c}} \Phi(a+1,\nu+1;\frac{1}{2c})},
\end{align*}
which implies (\ref{2e:sigma-minus-mellin}) in case $b>c$. The proof of (\ref{2e:sigma-minus-mellin}) in case $b<c$ is obtained similarly using
\begin{equation}
\label{2e:b-z}
\begin{split}
E^{(-\nu)}[(b-Z)_{+}^a] & = \int_{\frac{1}{2b}}^{\infty}
\frac{1}{\Gamma(\nu)} \Big(b-\frac{1}{2x}\Big)^a x^{\nu-1} e^{-x} \dd x \\
 & = \frac{b^{a-\nu}}{2^\nu \Gamma(\nu)} e^{-\frac{1}{2b}}
\Gamma(a+1)
\Psi(a+1,\nu+1;\frac{1}{2b}).
\end{split} \end{equation}
The proof is complete since the Mellin transform
(\ref{2e:sigma-minus-mellin}) in the interval $[0, \nu)$
characterizes the distribution of $\sigma_-$.
\end{proof}


\begin{rem} Let us note that by letting $a$ tend to $0$ in (\ref{2e:sigma-minus-mellin}), we obtain that
\begin{align}
 P^{(-\nu)}(\sigma_{-}<b) = \left\{
  \begin{array}{l l}
    \frac{b^{-\nu}e^{-\frac{1}{2b}}\int_0^\infty e^{-\frac{t}{2b}}(1+t)^{\nu-1}dt}
{c^{-\nu}e^{-\frac{1}{2c}}\int_0^\infty e^{-\frac{t}{2c}}(1+t)^{\nu-1}dt} & \quad \text{if $b < c$;} \label{eq1}\\[4pt]
     \frac{b^{-\nu}e^{-\frac{1}{2b}}\int_0^1 e^{\frac{t}{2b}}(1-t)^{\nu-1}dt}
{c^{-\nu}e^{-\frac{1}{2c}}\int_0^1 e^{-\frac{t}{2c}}(1-t)^{\nu-1}dt}& \quad \text{if $b> c$.}
  \end{array} \right.
\end{align}
In fact, $\tau$ is the first hitting time of $c$ by the  diffusion
$\eta^{(-\nu)}=\{\eta^{(-\nu)}_t\}_{t\geqq0}$ defined by \begin{equation}\label{auxiliary-diffusion}
\eta^{(-\nu)}_t = e^{-2B_t^{(-\nu)}} (b-A_t^{(-\nu)}),\; t\geqq 0,
\end{equation}
killed when it first hits $0$. A scale function and speed measure  of $\eta^{(-\nu)}$ are  given, respectively, by
\begin{equation*}
s(x)=\int_1^x y^{-\nu-1} e^{-\frac{1}{2y}} \dd y, \quad m(dx)= \frac12 x^{\nu-1}e^{\frac{1}{2x}}dx, \; x\in \mathbf{R}.
\end{equation*}
The boundary $\infty$ is natural and (\ref{eq1})  is easily checked using diffusion technics (see (\cite{RY}, Proposition (3.2), P.301), since
\begin{align*}
 P^{(-\nu)}(\sigma_{-}<b) =P^{(-\nu)}(\tau<+\infty)= \left\{
  \begin{array}{l l}
    \lim_{x\rightarrow 0}\frac{s(b)-s(x)}{s(b)-s(x)} & \quad \text{if $b < c$;}\\[4pt]
     \lim_{x\rightarrow \infty}\frac{s(x)-s(b)}{s(x)-s(c)}& \quad \text{if $b> c$,}
  \end{array} \right.
\end{align*}
with
\begin{equation*}
b^{-\nu}e^{-\frac{1}{2b}}\int_0^\infty e^{-\frac{t}{2b}}(1+t)^{\nu-1}dt
=\int_0^b \xi^{-\nu-1}e^{-\frac{1}{2\xi}}d\xi=s(b)-s(0)
\end{equation*}
and
\begin{equation*}
b^{-\nu}e^{-\frac{1}{2b}}\int_0^1 e^{\frac{t}{2b}}(1-t)^{\nu-1}dt
=\int_b^\infty \xi^{-\nu-1}e^{-\frac{1}{2\xi}}d\xi =\lim_{x\rightarrow \infty}(s(x)-s(b)).
\end{equation*}
\end{rem}

For completeness, we provide an explanation of what happens when we let $c$ tend to $0$ or $\infty$ in the equalities in distribution (\ref{rae-sigmaminus}) and (\ref{rae-sigmaminus-b-less-c}). For convenience, we write $\tau(c)$ and $\sigma_{-}(c)$, respectively, for $\tau$ and $\sigma_{-}$.

\begin{cor} \label{Verification-Limit} The convergence $\sigma_{-}(c)\rightarrow b\wedge \zeta$ holds $P^{(-\nu)}$ a.s. as $c\rightarrow \infty$.  Furthermore,
$(b-\sigma_{-})/c\rightarrow  R_{b}^2$ holds $P^{(-\nu)}$ a.s. as $c\rightarrow 0$.  As a consequence, there is the following identity in distribution
\begin{eqnarray}\label{By-product}
(Z-b)_+\stackrel{(law)}{=} ZR_{b}^2
\end{eqnarray}
where $Z$ is independent of $R$.
\end{cor}

\begin{proof}  As $c\rightarrow \infty$,  we have $
 \sigma_{-}(c)\rightarrow \inf\{s<b\; ;\; R_s=0\}=b\wedge \zeta$ since $\inf \emptyset=b$ in this case.  Next, because $\sigma_{-}(c)=A^{(-\nu)}_{\tau(c)}$ and  $\tau(c)\rightarrow C^{(-\nu)}_b$, we get that  $\sigma_{-}(c)\rightarrow b$ a.s., as $c\rightarrow 0$. Hence by continuity of $R$, we get that $R_{\sigma_{-}(c)}\rightarrow R_{b}$ a.s.
Formula (\ref{By-product}) follows from  (\ref{rae-sigmaminus}) by letting $c$ tend to 0.
\end{proof}
\begin{rem}
We present here yet another way of proving (\ref{By-product}). Recall that the semi-group of a Bessel process of index $-\nu$, $\nu>0$, when $0$ is a killing boundary,  is given by
\[p_t^{(-\nu)}(x,dy)=\frac{y}{t}\left(\frac{y}{x} \right)^{-\nu}e^{-\frac{x^2+y^2}{2t}}I_{\nu}(\frac{xy}{t})dy,
\]
see for example \cite{Borodin-Salminen}.
By using the expansion
\begin{equation}\label{Expansion} I_{\nu}(z)= \sum_{k=0}^{\infty} \frac{(z/2)^{\nu+2k}}{\Gamma(k+1)\Gamma(k+\nu+1)},
\end{equation}
found in (\cite{Lebedev}, P.108),
we can evaluate the positive real moments of $R_{b\wedge \zeta}^2$ to get
\begin{equation}\label{Moments-Bessel}
E\left[R_{b}^{2a} \right]=(2b)^{a-\nu}e^{-\frac{1}{2b}}\frac{\Gamma(a+1)}{\Gamma(\nu+1)}\phi(a+1, \nu+1, \frac{1}{2b}), \quad a<\nu.
\end{equation}
On the other hand, it is easy to see that
\begin{equation}\label{Moments-Gamma}
E\left[Z^a \right]=2^{-a}\frac{\Gamma(\nu-a)}{\Gamma(\nu)}, \quad a<\nu.
\end{equation}
Combining (\ref{2e:z-b}), (\ref{Moments-Bessel}) and (\ref{Moments-Gamma}),  yields that the Mellin transforms of the two sides of (\ref{By-product}) are equal.
\end{rem}

\begin{rem}\label{Remark-3} The results of Theorem  \ref{proposition-b-larger-than-c} and Corollary \ref{Verification-Limit} extend to positive spectrally one sided (spectrally negative when $b>c$ and spectrally positive when $b<c$) self-similar Markov processes with index 2. That is, we replace $B^{(-\nu)}$ by a L\'evy process $\xi:=(\xi_t, t\geq 0)$. Defining $A_{t}(\xi)=\int_0^te^{\xi_s}\, ds$, $t\geq 0$, we know that $A_{+\infty}(\xi)<\infty$ a.s. if and only if $\lim_{t\rightarrow \infty}t^{-1}\xi_t:=-\nu<0$, see (\cite{Bertoin-Yor-05}, Thm 1). Next, we replace  $R$ by the self-similar image $X$ of $\xi$ by the Lamperti transform (\ref{Lamperti}). Assuming that $b>c$ and $\xi$ is spectrally negative (thus $X$ also does not have positive jumps), consider the first hitting time $\sigma_{-}(X)$
 of the square-root boundary $s\mapsto \sqrt{(b- s)/c}$ with $\inf{\emptyset}=b\wedge\zeta(X)$ where  $\zeta(X)$ is the lifetime of $X$. Repeating the arguments of the proof of (\ref{rae-sigmaminus}), we see that an identity of that type holds, where $\sigma_{-}$ and $Z$ are replaced by $\sigma_{-}(X)$ and a copy $\tilde{A}_{\infty}(\xi)$ of $A_{\infty}(\xi)$ which is independent of $(\xi_s, s\geq 0)$, respectively.
By letting $c$ tends to $0$ (the trajectories of $X$ are not continuous but we can still get the limits in distribution),  we obtain the following generalization of (\ref{By-product}),
\begin{eqnarray}\label{By-product-general}
(A_{\infty}(\xi)-b)_+\stackrel{(law)}{=} X_{b\wedge \zeta(X)}^2 \tilde{A}_{\infty}(\xi).
\end{eqnarray}
  In the same spirit of the special functions introduced in \cite{Patie-09}, the analogue of (\ref{2e:z-b}) gives an extension of the confluent hypergeometric function $\Phi$ for non local type operators. The factorization of (\ref{By-product-general}) is of  different type than  the factorizations of the exponential distribution discovered in \cite{Bertoin-Yor-01}.  The study of these identities in distribution and their consequences, in the jumping setting, is a an interesting future research project.
\end{rem}
Recall the absolute continuity of the probability laws $P^{(\pm\nu)}$ of
 Bessel processes
\begin{equation} \label{2e:abs-conti}
\dd P_x^{(\nu)}\Big|_{\sF_t} =
\Big( \frac{R_{t\wedge\sigma_{-}}}{x} \Big)^{2\nu}
\dd P_x^{(-\nu)}\Big|_{\sF_t}
\end{equation}
found for instance in \cite{RY}. It follows, by cobining (\ref{2e:sigma-minus-mellin}) and (\ref{2e:abs-conti}), that
\begin{align*}
E^{(\nu)}\Big[\Big( 1-\frac{\sigma_{-}}{b} \Big)^a I_{\{\sigma_{-}<b\}}\Big]
 & = E^{(-\nu)}\Big[ (R_{\sigma_{-}})^{2\nu}
\Big( 1-\frac{\sigma_{-}}{b} \Big)^a I_{\{\sigma_{-}<b\}}\Big] \\
 & = E^{(-\nu)}\Big[ \Big( \frac{b-\sigma_{-}}{c} \Big)^{\nu}
\Big( 1-\frac{\sigma_{-}}{b} \Big)^a I_{\{\sigma_{-}<b\}}\Big] \\
 & = \frac{ e^{-\frac{1}{2b}} \Lambda(a+\nu+1,\nu+1;\frac{1}{2b})}
{ e^{-\frac{1}{2c}} \Lambda(a+\nu+1,\nu+1;\frac{1}{2c})}.
\end{align*}
It is obvious that if $b>c$ then $\sigma_{-}<\infty$ holds $P^{(\nu)}$ a.s. This is confirmed, by letting $a\downarrow 0$ in our calculations, because in this case  $\Phi(\nu+1,\nu+1;z)=e^z$.
Hence, we obtain the following result.

\begin{thm} \label{p:plus}
 For $a>0$, we have
\begin{equation}\label{eqn:theScaleFunction}
E^{(\nu)}\Big[\Big( 1-\frac{\sigma_{-}}{b} \Big)^a \Big] =
\frac{ e^{-\frac{1}{2b}} \Lambda(a+\nu+1,\nu+1;\frac{1}{2b})}
{ e^{-\frac{1}{2c}} \Lambda(a+\nu+1,\nu+1;\frac{1}{2c})},
\end{equation}
where $\Lambda=\Psi$ when $b<c$ and $\Lambda=\Phi$ when $b>c$. Note that, $P^{(\nu)}(\sigma_{-}<b)=1$ if $b>c$.
\end{thm}

\section{Characterization of the distribution of $\sigma_{+}$ reviewed}
Our aim is to establish a connection between the distributions of $\sigma_{-}$ and $\sigma_{+}$. We do this by combining our results for $\sigma_{-}$ with results in  \cite{alili-patie}. Then, this is applied to give a new probabilistic proof of formulae  \eqref{1e:by-yor-minus} and \eqref{1e:by-yor-plus}  in case $b>c$, which case was not dealt with in \cite{ESY}.
\begin{thm} \label{t:minus} The probability distributions of $\sigma_{-}$ and $\sigma_+$ under $P^{(\nu)}$ are related as follows
\begin{equation}\label{sigma-plus-sigma-minus}
P^{(\nu)}\Big(\frac{b}{b+\sigma_+}\in dt \Big)= t^{-\nu-1} e^{\frac{1}{2b}-\frac{1}{2c}}P^{(\nu)}\Big(\big(1-\frac{\sigma_-}{b}\big)\in dt \Big), \; t<1.
\end{equation}
As a consequence, formulae \eqref{1e:by-yor-minus} and \eqref{1e:by-yor-plus}   hold true.
\end{thm}
\begin{proof} We shall first prove that, for $a>\nu+1$, we have
\begin{equation}\label{random-affine-sigma-plus-minus-equivalent}
E^{(\nu)}\Big[ \Big( 1+\frac{1}{b}\sigma_+ \Big)^{-a} \Big] =
e^{\frac{1}{2b}-\frac{1}{2c}} E^{(\nu)}\Big[
\Big( 1-\frac{\sigma_{-}}{b} \Big)^{a-\nu-1} I_{\{\sigma_{-}<b\}}\Big]
\end{equation}
which, clearly, is equivalent to \eqref{sigma-plus-sigma-minus}.
We use the following relationship which is obtained from Theorem \ref{p:plus},
\begin{equation} \label{3e:start}
e^{\frac{1}{2b}-\frac{1}{c}} E^{(\nu)}\Big[
\Big( 1-\frac{\sigma_{-}}{b} \Big)^{a-\nu-1} \Big] =
\frac{\Lambda(a,\nu+1;\frac{1}{2b})}{\Lambda(a,\nu+1;\frac{1}{2c})}.
\end{equation}
\indent
To proceed, following \cite{alili-patie},
we introduce the probability measure $P_x^{(\nu,\beta)}, \beta\in\bR$, by
\begin{equation*}
\dd P_x^{(\nu,\beta)}\Big|_{\sF_t} = \frac{1}{(1+\beta t)^{\nu+1}}
e^{\frac{\beta R_t^2}{2(1+\beta t)}-\frac{\beta x^2}{2}}
\dd P_x^{(\nu)}\Big|_{\sF_{t\wedge\zeta^{(\beta)}}},
\end{equation*}
where, as before, $P_x^{(\nu)}$ is the probability law of
the Bessel processes with index $\nu$ starting from $x$ and
\begin{equation*}
\zeta^{(\beta)} = \begin{cases} 1/|\beta|, & \text{ if } \beta<0, \\
 +\infty, & \text{ if } \beta\geqq0. \end{cases}
\end{equation*}
We consider the path transform $S^{(\beta)}$ given by
\begin{equation*}
S^{(\beta)}(R)_t = (1+\beta t)R_{\frac{t}{1+\beta t}},\quad t<\zeta^{(\beta)}.
\end{equation*}
Then, it is shown in Lemma 3.4 of \cite{alili-patie} that
the induced measure of $P^{(\nu)}_x$ by $S^{(\beta)}$ is $P_x^{(\nu,\beta)}$.

We take $\beta=-\frac{1}{b}$.
Then, we have
\begin{equation*}
\frac{\beta(R_{\sigma_{-}})^2}{1+\beta\sigma_{-}} =
\frac{-\frac{1}{b} \frac{b-\sigma_{-}}{c}}{1-\frac{\sigma_{-}}{b}}
= -\frac{1}{c}
\end{equation*}
and
\begin{align*}
 e^{\frac{1}{2b}-\frac{1}{c}} E^{(\nu)}\Big[
\Big( 1-\frac{\sigma_{-}}{b} \Big)^{a-\nu-1} I_{\{\sigma_{-}<b\}} \Big]
 & = E^{(\nu)}\Big[
e^{\frac{\beta(R_{\sigma_{-}})^2}{2(1+\beta\sigma_{-})}-\frac{\beta}{2}}
\Big( 1-\frac{\sigma_{-}}{b} \Big)^{a-\nu-1} I_{\{\sigma_{-}<b\}} \Big] \\
 & = E^{(\nu,\beta)}\Big[
\Big( 1-\frac{\sigma_{-}}{b} \Big)^{a} I_{\{\sigma_{-}<b\}} \Big] .
\end{align*}
\indent
Since the probability law of $S^{(\beta)}(R^{(\nu)})$ is $P_1^{(\nu,\beta)}$,
the probability law of $\sigma_{-}$ under $P_1^{(\nu,\beta)}$ is that of
\begin{equation*}
\wt{\sigma} := \inf\Big\{s>0 \; ; \;
(1+\beta s)^2 (R_{\frac{s}{1+\beta s}})^2 =\frac{b-s}{c} \Big\}
\end{equation*}
under $P_1^{(\nu)}$ and we have
\begin{equation*}
e^{\frac{1}{2b}-\frac{1}{c}} E^{(\nu)}\Big[
\Big( 1-\frac{\sigma_{-}}{b} \Big)^{a-\nu-1} I_{\{\sigma_{-}<b\}} \Big] =
E^{(\nu)}\Big[ \Big( 1-\frac{\wt{\sigma}}{b} \Big)^a
I_{\{\sigma_{-}<b\}} \Big].
\end{equation*}
Moreover, noting $\frac{b-s}{c}=\frac{b}{c}(1+\beta s)$, we see
\begin{equation*}
\wt{\sigma} = \inf\Big\{s>0 \; ; \; (R_{\frac{s}{1+\beta s}})^2 =
\frac{b}{c}\Big( 1-\frac{\beta s}{1+\beta s} \Big) \Big\}.
\end{equation*}
Since
\begin{equation*}
\sigma_{+}  = \inf\Big\{ u>0\; ; \; (R_u)^2=\frac{b}{c}(1-\beta u)\Big\},
\end{equation*}
we obtain
\begin{equation*}
\frac{\wt{\sigma}}{1+\beta\wt{\sigma}}=\sigma_{+} \qquad \text{and} \qquad
\wt{\sigma}=\frac{\sigma_+}{1-\beta\sigma_+}=\frac{\sigma_+}{1+\frac{\sigma_+}{b}}<b.
\end{equation*}
Hence, we have
\begin{equation*} \begin{split}
 E^{(\nu)}\Big[ \Big( 1-\frac{\wt{\sigma}}{b} \Big)^a
I_{\{\wt{\sigma}<b\}} \Big] & =
E^{(\nu)}\Big[ \Big( 1+\beta \frac{\sigma}{1-\beta\sigma} \Big)^{a} \Big]\\
 & = E^{(\nu)}\Big[ \Big( 1+\frac{1}{b}\sigma \Big)^{-a} \Big].
\end{split} \end{equation*}
\noindent Thus, we have proved \eqref{random-affine-sigma-plus-minus-equivalent}. Combining this with \eqref{3e:start},
we obtain \eqref{1e:by-yor-plus}.

Formula \eqref{1e:by-yor-minus} follows from \eqref{1e:by-yor-plus}
and the absolute continuity \eqref{2e:abs-conti}.
In fact, we deduce from them
\begin{align*}
E^{(-\nu)}[(b+\sigma_{+})^{-a}] & =
E^{(\nu)}[(R_{\sigma_{+}})^{-2\nu} (b+\sigma_{+})^{-a}]\\
& =
E^{(\nu)}\Big[ \Big( \frac{b+\sigma_{+}}{c} \Big)^{-\nu}
(b+\sigma_{+})^{-a} \Big] \\
 & = c^\nu b^{-a-\nu}
\frac{\Lambda(a+\nu,\nu+1;\frac{1}{2b})}{\Lambda(a+\nu,\nu+1;\frac{1}{2c})},
\end{align*}
which is exactly \eqref{1e:by-yor-minus}.
\end{proof}
Our aim now is to establish the random factorizations satisfied by $\sigma_{+}$ under $P^{(-\nu)}$. For completeness, we include the case  $b<c$  which was  treated in \cite{ESY}, where \eqref{random-affine-sigma-plus-minus} was first proved.
\begin{thm} Let $\sigma^*_+$ be equal to $\sigma_+$ conditioned on $\sigma_+<\infty$. Then, under $P^{(-\nu)}$ we have the following identities in distribution which characterize the distribution of $\sigma_{+}$.
\begin{itemize}
\item[(1)]
 If $b<c$ then
\begin{equation}\label{random-affine-sigma-plus-minus}
1+\frac{Z}{b} \eil \big(1+\frac{\sigma_{+}}{b}\Big) \Big( 1 + \frac{Z}{c} \Big)
\end{equation}
where $\sigma_{+}$ and $Z$ on the right-hand side are independent.
\item[(2)]
If $b>c$  then
\begin{equation}\label{Identity-sigma-plus}
\frac{Z(b)}{b}-1\stackrel{(law)}{=}\Big(1+\frac{\sigma_+^*}{b}\Big)^{-1}\Big(\frac{Z(c)}{c}-1\Big)
\end{equation}
where $Z(\alpha)$, for $\alpha>0$, is a random variable with distribution
\begin{equation*}P^{(-\nu)}\Big(Z(\alpha)\in dz\Big)=\frac{{(z-\alpha)^{\nu-1}}}{E^{(-\nu)}\Big(\big(Z-\alpha\big)^{\nu-1}; Z>\alpha\Big)}P^{(-\nu)}(Z\in dz), z>\alpha,
\end{equation*}
and $Z(c)$ and $\sigma^*_{+}$ on the right hand side are independent.
\end{itemize}
\end{thm}
\begin{proof} We refer to \cite {ESY} for a proof  of \eqref{random-affine-sigma-plus-minus}. To prove (\ref{Identity-sigma-plus}), observe that combining \eqref{1e:by-yor-minus} and  (\ref{2e:z-b}) we obtain for $a<1$,
\begin{eqnarray*}
E^{(-\nu)}\Big[\big(b+\sigma_+\big)^{-a}\Big]E^{(-\nu)}\Big[\big(Z-c\big)^{\nu+a-1}, Z>c\Big]=\frac{c^{\nu+a-1}}{b^{\nu+2a-1}}e^{\frac{1}{2b}-\frac{1}{2c}}E^{(-\nu)}\Big[\big(Z-b\big)^{\nu+a-1}, Z>b\Big].
\end{eqnarray*}
Now, letting $a$ tend to $0$ yields
\begin{equation*}
P\big(\sigma_{+}<\infty\big)E^{(-\nu)}\Big[\big(Z-c\big)^{\nu-1}, Z>c\Big]=\frac{c^{\nu-1}}{b^{\nu-1}}e^{\frac{1}{2b}-\frac{1}{2c}}E^{(-\nu)}\Big[\big(Z-b\big)^{\nu-1}, Z>b\Big].
\end{equation*}
Hence,
\begin{eqnarray*}
E^{(-\nu)}\Big[\big(b+\sigma_+\big)^{-a}| \sigma_{+}<\infty\Big]\frac{E^{(-\nu)}\Big[\big(Z-c\big)^{\nu+a-1}, Z>c \Big]}{E^{(-\nu)}\Big[\big(Z-c\big)^{\nu-1}, Z>c\Big]}=\frac{c^{a}}{b^{2a}}\frac{E^{(-\nu)}\Big[\big(Z-b\big)^{\nu+a-1}, Z>b\Big]}{E^{(-\nu)}\Big[\big(Z-b\big)^{\nu-1}, Z>b\Big]}.
\end{eqnarray*}
We get (\ref{Identity-sigma-plus})
by the injectivity of Mellin transform.  Equations \eqref{random-affine-sigma-plus-minus} and \eqref{Identity-sigma-plus} imply \eqref{1e:by-yor-minus} and hence characterize the distribution of $\sigma_+$.
\end{proof}

\noindent{\bf Acknowledgement.}\
We are grateful to the anonymous referees whose suggestions helped improving this paper. 
A part of this work has been done during the second author's visit
to Mathematics Institute, University of Warwick, in 2017.
He thanks Professor David Elworthy for everything on his stay.


\end{document}